\def\version{12/02/2012 Version 6 \hfill arXiv:1105.0692}

\documentclass[11pt,a4paper]{amsart}
\usepackage{fullpage}
\usepackage{amssymb,amscd}
\usepackage[all]{xy}
\usepackage{mathrsfs}
\usepackage[numeric,lite]{amsrefs}
\usepackage{hyperref}
\usepackage{shuffle}

\theoremstyle{plain}
\newtheorem{thm}{Theorem}[section]
\newtheorem{lem}[thm]{Lemma}
\newtheorem{prop}[thm]{Proposition}
\newtheorem{cor}[thm]{Corollary}

\theoremstyle{definition}
\newtheorem{rem}[thm]{Remark}
\newtheorem{defn}[thm]{Definition}
\newtheorem{examp}[thm]{Example}

\numberwithin{equation}{section}

\def\ie{\emph{i.e.}}

\def\:{\colon}
\def\.{\cdot}
\def\<{\left\langle}
\def\>{\right\rangle}
\def\({\left(}
\def\){\right)}
\def\ph#1{\phantom{#1}}
\def\epsilon{\varepsilon}

\def\leq{\leqslant}
\def\geq{\geqslant}

\def\lra{\longrightarrow}
\def\Lra{\Longrightarrow}
\def\ra{\rightarrow}

\def\tilde#1{\widetilde{#1}}
\def\iso{\cong}
\def\phi{\varphi}

\DeclareMathOperator{\rank}{rank}

\def\C{\mathbb{C}}

\def\F{\mathbb{F}}

\def\H{\mathbb{H}}
\def\k{\Bbbk}

\def\Q{\mathbb{Q}}
\def\N{\mathbb{N}}
\def\R{\mathbb{R}}

\def\Z{\mathbb{Z}}

\DeclareMathOperator{\Cotor}{Cotor}

\DeclareMathOperator{\Tor}{Tor}

\DeclareMathOperator{\Id}{Id}

\def\CP{\C P}
\def\HP{\H P}

\DeclareMathOperator{\Maps}{Maps}
\DeclareMathOperator{\ev}{ev}
\def\Sq{\mathrm{Sq}}
\def\P{\mathcal{P}}
\def\SO{\mathrm{SO}}
\def\Spin{\mathrm{Spin}}

\DeclareMathOperator{\Sgn}{Sgn}

\begin{document}
\title[Cohomology of loop spaces for Thom spaces]
{On the cohomology of loop spaces for some Thom spaces}
\author{Andrew Baker}
\address{School of Mathematics \& Statistics, University
of Glasgow, Glasgow G12 8QW, Scotland.}
\email{a.baker@maths.gla.ac.uk}
\urladdr{http://www.maths.gla.ac.uk/$\sim$ajb}

\thanks{The author thanks Nigel Ray and Birgit Richter for
much help and encouragement, Larry Smith for pointing out
the related work of Petrie, and Teimuraz Pirashvilli who
drew our attention to the classic paper of Bott and Samelson
which predated the James splitting.}

\date{\version}
\begin{abstract}
In this paper we identify conditions under which the cohomology
$H^*(\Omega M\xi;\k)$ for the loop space $\Omega M\xi$ of the
Thom space $M\xi$ of a spherical fibration $\xi\downarrow B$
can be a polynomial ring. We use the Eilenberg-Moore spectral
sequence which has a particularly simple form when the Euler
class $e(\xi)\in H^n(B;\k)$ vanishes, or equivalently when an
orientation class for the Thom space has trivial square. As
a consequence of our homological calculations we are able to
show that the suspension spectrum $\Sigma^\infty\Omega M\xi$
has a local splitting replacing the James splitting of
$\Sigma\Omega M\xi$ when $M\xi$ is a suspension.
\end{abstract}

\maketitle

\section*{Introduction}

In \cite{AB&BR:qsymm}, topological methods were used to prove the
algebraic Ditter's conjecture on quasi-symmetric functions, which
is equivalent to the assertion that $H^*(\Omega\Sigma\CP^\infty;\Z)$
is a polynomial ring (infinitely generated but of finite type). Most
of the ingredients of the proof given there are essentially formal
within algebraic topology, the exception being James's splitting of
$\Sigma\Omega\Sigma\CP^\infty$. The purpose of this paper is to
identify circumstances in which the cohomology $H^*(\Omega M\xi;\k)$
of the loop space $\Omega M\xi$ of the Thom space $M\xi$ of a spherical
fibration $\xi\downarrow B$ can be a polynomial ring. In place of
the James splitting we use the Eilenberg-Moore spectral sequence
which has a particularly simple form when the Euler class
$e(\xi)\in H^n(B;\k)$ vanishes, or equivalently when an orientation
class for the Thom space has trivial square. As a consequence of our
homological calculations we are able to show that the suspension
spectrum $\Sigma^\infty\Omega M\xi$ has a local splitting generalizing
that for $\Sigma\Omega M\xi$ when $M\xi$ is a suspension. Our results
appear to be more general and essentially formal in that only generic
properties of the Eilenberg-Moore spectral sequence are used; however,
the above stable splitting is a weaker result than the James splitting.

Although our examples are all associated with vector bundles, our methods
are valid for arbitrary spherical fibrations, and even more generally
they apply to $p$-local or $p$-complete spherical fibrations. We hope
to consider examples associated with $p$-compact groups in future work.

We were very influenced by the discussion of the cohomology of
$\Omega\Sigma X$ in Smith's article~\cite{LS:Survey}. Massey's
paper~\cite{Massey:CohSphBdles} provides a useful background to our
work. Although we do not make direct use of it, Ray's
paper~\cite{NR:LoopsCones} has ideas that might allow generalizations
to other mapping cones. Although we do not make direct use of the results
of these papers, we remark that Bott \& Samelson~\cite{RB-HS:Pontryagin}
and Petrie~\cite{TP:CohLoopSpces} gave earlier versions of the arguments
we use, however neither paper contains the full range of our results; in
particular the latter does not deal with questions about multiplicative
structure.

\section{Thom complexes of spherical fibrations}\label{sec:ThomComp}

Let $B$ be space and let $\xi\:S^{n-1}\lra S\lra B$
be a spherical fibration with associated disc bundle
$D^n\lra D\lra B$. The Thom space $M=M\xi$ is the
cofibre of the inclusion $S\lra D$, \ie, the quotient
space $D/S$. In each fibre this corresponds to the
inclusion $S^{n-1}\lra D^n$ and there is a cofibre
sequence of based spaces
\begin{equation}\label{eq:M-cofibseq}
S_+ \lra D_+ \lra M \xrightarrow{\delta} \Sigma S_+.
\end{equation}
Here we implicitly allow for generalizations to include
localized spheres as fibres and bundles with structure
monoids obtained from the invertible components of
$\Maps(S^{n-1},S^{n-1})$.

We are interested in the based loop space $\Omega M$.
There is an obvious unbased map $S\lra\Omega M$ which
sends $v\in S_b$ (the fibre above $b\in B$) to the
non-constant loop $[0,1]\lra M$ given by $t\mapsto[(2t-1)v]$,
running through $b$ parallel to~$v$ and passing through
the base point at times $t=0,1$. This extends to a based
map $\theta\:S_+\lra\Omega M$. We write $\ev\:\Sigma\Omega M\lra M$
for the evaluation map. See~\cite{NR:LoopsCones} for
a related construction

Our next result is surely standard but we don't know
an explicit reference.
\begin{lem}\label{lem:M->M}
The composition
\[
M \xrightarrow{\;\delta\;} \Sigma S_+
  \xrightarrow{\;\Sigma\theta\;} \Sigma\Omega M
  \xrightarrow{\;\ev\;} M
\]
is a homotopy equivalence.
\end{lem}
\begin{proof}
This follows by unravelling definitions. Depending
on the sign conventions used for the coboundary map
of a cofibration, it is homotopic to $\pm\Id$.
\end{proof}
\begin{cor}\label{cor:M->M}
Let $h^*(-)$ be a reduced cohomology theory. Then
the cohomology suspension map
\[
h^*(M) \xrightarrow{\;\ev^*\;} h^*(\Sigma\Omega M)
       \xrightarrow{\;\iso\;} h^{*-1}(\Omega M)
\]
is a monomorphism.
\end{cor}

These two results are analogues of results for a
suspension $\Sigma X$ in~\cite{LS:Survey}*{section~2}
which depend on the fact that $\Sigma,\Omega$ is an
adjoint pair.

The next result is standard, although it seems to be
hard to find it stated in this form in the literature,
see for example~\cite{TP:CohLoopSpces}*{section~1}.
To clarify what is involved, we give details. First
recall an algebraic notion.

Let $\k$ be a commutative unital ring; tensor products
will be taken over $\k$ unless otherwise specified. Let
$A$ be a commutative unital graded $\k$-algebra with
product $\phi\:A\otimes A \lra A$.
\begin{defn}\label{defn:nonunitalalg}
A \emph{non-unital $A$-algebra} is a left $A$-module $M$
with multiplication
\[
A\otimes M \lra M; \quad a\otimes m \mapsto a\.m
\]
and a non-unital associative product $\mu\:M\otimes_A M\lra M$.
Thus the following diagram commutes, where
$\mathrm{T}\:M\otimes A\lra A\otimes M$ is the switch map
with appropriate signs based on gradings.
\[
\xymatrix{
A\otimes M\otimes A \otimes M \ar[rr]^{I\otimes\mathrm{T}\otimes I}\ar[d]_{\.\otimes\.}
 & & A\otimes A \otimes M\otimes M\ar[d]^{\phi\otimes\mu} \\
M\otimes M\ar[dr]_{\mu} & & A\otimes M\ar[dl]^{\.} \\
& M &
}
\]
For homogeneous elements $a_1,a_2\in A$, $m_1,m_2\in M$
and $m_1m_2=\mu(m_1\otimes m_2)$,
\[
(a_1a_2)\.(m_1m_2) =
   (-1)^{|a_2|\,|m_1|} \mu((a_1\.m_1)\otimes(a_2m_2)).
\]
\end{defn}

There is a Thom diagonal map $\tilde\Delta\:M\lra B_+\wedge M$
fitting into a strictly commutative diagram
\begin{equation}\label{eq:ThomDiagonal}
\xymatrix{
D_+\ar[rr]^(.45){\Delta}\ar[d]_{\mathrm{quot.}}
                  && D_+\wedge D_+\ar[d]^{\mathrm{quot.}} \\
M \ar[rr]^(.45){\tilde\Delta} && B_+\wedge M
}
\end{equation}
whose vertical maps are the evident quotient maps. If
$h^*(-)$ is a multiplicative cohomology theory, then
$\tilde\Delta$ induces an external product
\[
\cdot\;\:h^*(B)\otimes\tilde{h}^*(M)
\xrightarrow{\;\ph{\tilde\Delta^*}\;} \tilde{h}^*(B_+\wedge M)
\xrightarrow{\;\;\tilde\Delta^*\;} \tilde{h}^*(M);
\quad
b\otimes m\mapsto b\cdot m,
\]
where $\tilde{h}^*(-)$ denotes the reduced theory.
\begin{thm}\label{thm:Thomspce-Mult}
Suppose that\/ $h^*(-)$ is a commutative multiplicative
cohomology theory. Then the external product induced
from $\tilde\Delta$ makes $\tilde{h}^*(M)$ into a left
$h^*(B)$-module enjoying the following properties. \\
\emph{(a)} If $M$ has an orientation $u\in\tilde h^n(M)$
then the associated Thom isomorphism
\[
h^*(B) \xrightarrow{\;\iso\;} \tilde h^*(M);
\quad
x\leftrightarrow x\.u
\]
makes $\tilde h^*(M)$ into a free $h^*(B)$-module of
rank~$1$. \\
\emph{(b)} The cup product on $\tilde h^*(M)$ makes
it a commutative non-unital $h^*(B)$-algebra. \\
\emph{(c)} When $h^*(-)= H^*(-;\F_p)$ for a prime~$p$,
the mod~$p$ Steenrod algebra acts compatibly so that
the Cartan formula holds for products of the form $t\.w$
with $t\in H^*(B;\F_p)$ and $w\in\tilde H^*(M;\F_p)$.
\end{thm}
\begin{proof}
The main point is to verify that the following diagram
commutes, where $\Delta$ always denotes an internal
based diagonal map $X\lra X\wedge X$.
\begin{equation}\label{eq:Thomspce-Mult}
\xymatrix{
&&M\ar[lld]_{\Delta}\ar[rrd]^{\tilde\Delta}&& \\
M\wedge M\ar[rd]_(.3){\ph{\Delta}\tilde\Delta\wedge\tilde\Delta}
              &&&& B_+\wedge M\ar[ld]^(.3){\Delta\wedge\Delta} \\
&B_+\wedge M\wedge B_+\wedge M \ar[rr]_{\ph{acd}\mathrm{switch}}
   &&B_+\wedge B_+\wedge M\wedge M &
}
\end{equation}
Making use of the commutative diagram~\eqref{eq:ThomDiagonal},
this follows from properties of the diagonal
$\Delta\:D_+\lra D_+\wedge D_+$ which is (strictly) coassociative,
cocommutative and counital (the counit is the projection $D_+\lra S^0$).
The diagram
\[
\xymatrix{
&&D_+\ar[lld]_{\Delta}\ar[rrd]^{\Delta}&& \\
D_+\wedge D_+\ar[rd]_(.4){\ph{\Delta}\Delta\wedge\Delta}
                  &&&& D_+\wedge D_+\ar[ld]^(.4){\Delta\wedge\Delta} \\
&D_+\wedge D_+\wedge D_+\wedge D_+\ar[rr]_{\ph{ac}\mathrm{switch}}
   &&D_+\wedge D_+\wedge D_+\wedge D_+&
}
\]
commutes, so by passing to the diagram of quotients we obtain
commutativity of~\eqref{eq:Thomspce-Mult}.

Applying $h^*(-)$ and $\tilde h^*(-)$ now give the algebraic
properties asserted. Of course $h^*(M)$ is also a commutative
unital $h^*$-algebra.

The statement about the Steenrod action follows from the
Cartan formula for external smash products and naturality.
\end{proof}
\begin{cor}\label{cor:Thomspce-Mult}
If the orientation $u$ satisfies $u^2=0$, then the product
in\/ $\tilde h^*(M)$ is trivial.
\end{cor}

Notice that the condition $u^2=0$ for one orientation implies
that the same is true for any orientation.

We end with another result involving the external diagonal.
\begin{lem}\label{lem:Diag-S/M}
The following diagram commutes.
\[
\xymatrix{
& \Sigma S_+ \ar[r]^{\;\Sigma\theta\;}\ar[dl]_{\Sigma\Delta}
    & \Sigma\Omega M\ar[r]^{\;\ev} & M\ar[d]^{\tilde\Delta}  \\
\Sigma S_+\wedge S_+\ar[d] &&& B_+\wedge M \\
\Sigma S_+\wedge B_+\ar[r]^{\iso}
       & B_+\wedge\Sigma S_+\ar[rr]^{\;\Id\wedge\Sigma\theta\;}
       & &   B_+\wedge\Sigma\Omega M\ar[u]_{\Id\wedge\ev\;}
}
\]
Hence if\/ $h^*(-)$ is a multiplicative cohomology theory,
then $(\ev\circ\Sigma\theta)^*\:\tilde h^*(M)\lra h^*(S)$
is a homomorphism of\/ $h^*(B)$-modules.
\end{lem}

\section{Recollections on the Eilenberg-Moore spectral sequence}
\label{sec:EMSS}

There is of course an extensive literature on Eilenberg-Moore
spectral sequence, but for our purposes most of what we need
can be found in Smith's excellent survey article~\cite{LS:Survey},
together with Rector and Smith's papers on Steenrod
operations~\cites{DLR:EMSS-StOps,LS:KunnethThm-I}. For the
homological algebra background and construction, see~\cite{LS:HomAlg&EMSS}.
Other useful sources
are~\cites{WGD:EMSS-Stcgce,LS:LNM134,LS:EMSS-construction,WMS:Book}.

In the following we will assume that $\k$ is a field, and
$H^*(-)=H^*(-;\k)$. We will also assume that our Thom
space~$M$ from Section~\ref{sec:ThomComp} has an orientation
in $H^*(-)$, $M$ is simply connected, and $H^*(B)$ has finite
type; these conditions are needed for convergence of the
Eilenberg-Moore spectral sequence
we will use.
\begin{thm}\label{thm:EMSS}
There is a second quadrant Eilenberg-Moore spectral sequence
of\/ $\k$-Hopf algebras $(\mathrm{E}_r^{*,*},d_r)$ with
differentials
\[
d_r\:\mathrm{E}_r^{s,t}\lra \mathrm{E}_r^{s+r,t-r+1}
\]
and
\[
\mathrm{E}_2^{s,t} = \Tor^{s,t}_{H^*(M)}(\k,\k)
                               \Lra H^{s+t}(\Omega M).
\]
\end{thm}
The grading conventions here give
\[
\Tor^{s,*}_{H^*(M)} = \Tor_{-s,*}^{H^*(M)}
\]
in the standard homological grading.

When $\k=\F_p$ for a prime $p$, this spectral sequence
admits Steenrod operations;
see~\cites{DLR:EMSS-StOps,WMS:Book,LS:EMSS-construction,LS:KunnethThm-I,LS:LNM134}.
We denote the mod~$p$ Steenrod algebra by $\mathcal{A}(p)^*$
or $\mathcal{A}^*$ when the prime~$p$ is clear.
\begin{thm}[]\label{thm:EMSS-p}
If $H^*(-)=H^*(-;\F_p)$ for a prime~$p$, the Eilenberg-Moore
spectral sequence is a spectral sequence of $\mathcal{A}^*$-Hopf
algebras.
\end{thm}

We will need explicit formulae for the Steenrod action. The
main result is the following.
\begin{prop}\label{prop:EMSS-Steenrodaction}
Suppose that $X$ is a based space. Then in the Eilenberg-Moore
spectral sequence
\[
\mathrm{E}_2^{*,*} = \Tor^{*,*}_{H^*(X;\F_p)}(\F_p,\F_p)
                   \Lra H^*(\Omega X;\F_p)
\]
the action of the Steenrod operations on the $\mathrm{E}_2$-term
is given in terms of the cobar construction by
\begin{align*}
\Sq^s[x_1|\cdots|x_n]
  &= \sum_{s_1+\cdots+s_n=s}[\Sq^{s_1}x_1|\cdots|\Sq^{s_n}x_n]
  &&\text{if $p=2$}, \\
\mathcal{P}^s[x_1|\cdots|x_n]
  &= \sum_{s_1+\cdots+s_n=s}[\mathcal{P}^{s_1}x_1|\cdots|\mathcal{P}^{s_n}x_n]
  &&\text{if $p$ is odd}.
\end{align*}
\end{prop}
\begin{proof}[Sketch of Proof]
There is a construction of the Eilenberg-Moore spectral
sequence for the pullback of a fibration $q$ along a
map~$f$.
\[
\xymatrix{
E'\ar[r]\ar[d]_{q'}
                                        & E\ar[d]^{q} \\
B'\ar[r]_{f} & B\ar@{}[ul]|<<<{\text{\Large$\lrcorner$}}
}
\]
For details see~\cites{WGD:EMSS-Stcgce,LS:KunnethThm-I}.
This approach involves the cosimplicial space $C^\bullet$
with
\[
C^s = E\times B^{\times s} \times B'
\]
and structure maps $h_t\:C^s\lra C^{s+1}$ ($0\leq t\leq s+1$),
\[
h_t(e,b_1,\ldots,b_s,b') =
\begin{cases}
(e,h(e),b_1,\ldots,b_s,b') & \text{if $t=0$}, \\
(e,b_1,\ldots,b_{t-1},b_t,b_t,b_{t+1},\ldots,b_s,b')
                       & \text{if $1\leq t\leq s$}, \\
(e,b_1,\ldots,b_s,q(b'),b') & \text{if $t=s+1$}.
\end{cases}
\]
The geometric realisation $|C^\bullet|$ admits a map
$E'\lra|C^\bullet|$, and on applying $H^*(-;\F_p)$ to
the coskeletal filtration of $|C^\bullet|$ we obtain
the Eilenberg-Moore spectral sequence for $H^*(E';\F_p)$.
Then the $\mathrm{E}_1$-term can be identified with bar
construction on $H^*(B;\F_p)$ and comes from the cohomology
of the filtration quotients which are suspensions of the
spaces $E\wedge B^{(s)}\wedge B'$. The action of Steenrod
operations on $\tilde H^*(E\wedge B^{(s)}\wedge B';\F_p)$
is determined using the Cartan formula, and gives the claimed
formulae in the $\mathrm{E}_2$-term.
\end{proof}

Now we come to a special situation that is our main concern.
\begin{thm}\label{thm:EMSS-u^2=0}
Suppose that the orientation $u\in H^n(M)=H^n(M;\k)$ satisfies
$u^2=0$. Then there is an isomorphism of Hopf algebras
\[
\Tor^{*,*}_{H^*(M)}(\k,\k) = \mathrm{B}^*(H^*(M)),
\]
where $\mathrm{B}^*(H^*(M))$ denotes the bar construction with
\[
\mathrm{B}^{-s}(H^*(M)) = (\tilde H^*(M))^{\otimes s}
\]
for $s\geq0$. The coproduct
\[
\psi\:\mathrm{B}^{-s}(H^*(M))\lra
\bigoplus_{i=0}^s \mathrm{B}^{-i}(H^*(M))\otimes\mathrm{B}^{i-s}(H^*(M))
\]
is the usual one with
\[
\psi([u_1|\cdots|u_s])
  = \sum_{i=0}^s\; [u_1|\cdots|u_i]\otimes[u_{i+1}|\cdots|u_s],
\]
where we use the traditional bar notation
$[w_1|\cdots|w_r]=w_1\otimes\cdots\otimes w_r$.
\end{thm}
\begin{proof}
The proof is identical to that for the case of $\Sigma X$
in~\cite{LS:Survey}*{section~2, example~4}, and uses the
fact that $\tilde H^*(N)$ has only trivial products by
Corollary~\ref{cor:Thomspce-Mult}.
\end{proof}
\begin{rem}\label{rem:EMSS-product}
The product in the $\mathrm{E}_2$-term is the shuffle product,
\[
[u_1|\cdots|u_r]\shuffle[v_1|\cdots|v_s] =
\sum_{\text{$(r,s)$ shuffles $\sigma$}} (-1)^{\Sgn(\sigma)}
             [w_{\sigma(1)}|w_{\sigma(2)}|\cdots|w_{\sigma(r+s)}],
\]
where $\sigma\in\Sigma_{r+s}$ is an \emph{$(r,s)$-shuffle} if
\[
\sigma(1)<\sigma(2)<\cdots<\sigma(r),
\quad
\sigma(r+1)<\sigma(r+2)<\cdots<\sigma(r+s),
\]
\[
w_{\sigma(i)} =
\begin{cases}
u_{\sigma(i)} & \text{if $1\leq \sigma(i)\leq r$}, \\
v_{\sigma(i)-r} & \text{if $r+1\leq \sigma(i)\leq r+s$},
\end{cases}
\]
and
\[
\Sgn(\sigma) = \sum_{(i,j)}(\deg w_i+1)(\deg w_{r+j}+1))
\]
where the summation is over pairs $(i,j)$ for which
$\sigma(i)>\sigma(r+j)$.
\end{rem}

In the situation of this Theorem we have
\begin{cor}\label{cor:EMSS-u^2=0}
The Eilenberg-Moore spectral sequence of\/ \emph{Theorem~\ref{thm:EMSS}}
collapses at the\/ $\mathrm{E}_2$-term.
\end{cor}

The proof is similar to that of~\cite{LS:Survey}*{section~2, example~4},
and depends on two observations on this spectral sequence for $H^*(\Omega M)$
under the conditions of Theorem~\ref{thm:EMSS}.
\begin{lem}\label{lem:EMSS-edgehomo}
The edge homomorphism $e\:\mathrm{E}_2^{-1,*+1}\lra H^*(\Omega M)$ can 
be identified with the composition
\[
H^{*+1}(M)\xrightarrow{\ev^*} H^{*+1}(\Sigma\Omega M)
          \xrightarrow{\iso}H^*(\Omega M)
\]
using the canonical isomorphism $\mathrm{E}_2^{-1,*+1}\xrightarrow{\iso}H^{*+1}(M)$.
\end{lem}
\begin{cor}\label{cor:EMSS-edgehomo}
The edge homomorphism $e\:\mathrm{E}_2^{-1,*+1}\lra H^*(\Omega M)$
is a monomorphism.
\end{cor}
\begin{proof}
This follows from Lemma~\ref{lem:M->M} since $(\Sigma\theta\circ\delta)^*$
provides a left inverse for~$e$.
\end{proof}

\section{On the cohomology of sphere bundles}\label{sec:CohSphBdles}

In this section we recall some results of Massey~\cite{Massey:CohSphBdles}*{part~II}.
We continue to use the notation and general set-up of Section~\ref{sec:ThomComp}.

We assume that our spherical fibration $\xi$ is orientable in
$H^*(-)=H^*(-;\k)$. Choosing an orientation class $u\in H^n(M)$,
we also suppose that $u^2=0$. Then~\eqref{eq:M-cofibseq} induces
an exact sequence
\[
0\ra H^*(B) \xrightarrow{\ph{\;\delta^*\;}} H^*(S)
            \xrightarrow{\;\delta^*\;} \tilde H^{*+1}(M) \ra 0
\]
in which $\delta^*$ is a an $H^*(B)$-module homomorphism with respect
to the obvious module structure on $H^*(S)$ and the Thom module
structure on $\tilde H^*(M)$. Since the left hand map is a monomorphism
we regard $H^*(B)$ as a subring of $H^*(S)$.

Now choose $v\in H^{n-1}(S)$ so that $\delta^*(v)=u$.
Then by~\cite{Massey:CohSphBdles}*{(8.1)} there is
a relation of the form
\begin{equation}\label{eq:WSM(8.1)}
v^2 = s + tv,
\end{equation}
where $s\in H^{2n-2}(B)$ and $t\in H^{n-1}(B)$. If
we make a different choice $v'\in H^{n-1}(S)$ with
$\delta^*(v')=u$, then $w=v'-v\in H^{n-1}(B)$ and
we find that
\[
(v')^2 = s' + t'v',
\]
where
\begin{align*}
s' &= s - wt - w^2, \\
t' &=
\begin{cases}
t & \text{if $n$ is even}, \\
t+2w & \text{if $n$ is odd}.
\end{cases}
\end{align*}
Massey also shows that when $n$ is odd and $\k=\F_2$,
\begin{equation}\label{eq:t=w(n-1)}
t = w_{n-1}(\xi).
\end{equation}
Here we define the Stiefel-Whitney class through the Wu
formula in $H^*(M)$,
\[
w_{n-1}(\xi)\.u = \Sq^{n-1}u.
\]
Of course this makes sense for any spherical fibration,
not just those associated with vector bundles.

Here are two examples that we will discuss again later.
\begin{examp}\label{examp:Spin(2/3)}
Consider the universal $\Spin(2)$ and $\Spin(3)$ bundles
$\zeta_2\downarrow B\Spin(2)$ and $\zeta_3\downarrow B\Spin(3)$
obtained from the canonical representations into $\SO(2)$
and $\SO(3)$. Of course the bases of these bundles can be
taken to be
\[
B\Spin(2) = \CP^\infty,
\quad
B\Spin(3) = \HP^\infty,
\]
and $\zeta_2=\eta^2$, the square of the universal complex
line bundle $\eta\downarrow\CP^\infty$. Since there are 
$\Spin(3)$-equivariant homeomorphisms
\[
\Spin(3)/\Spin(2) \iso \SO(3)/\SO(2) \iso S^2,
\]
the sphere bundle of $\zeta_3$
\[
E\Spin(3)/\Spin(2)\xrightarrow{\doteq}
E\Spin(3)\times_{\Spin(3)}\Spin(3)/\Spin(2)\xrightarrow{\ph{\doteq}}
E\Spin(3)/\Spin(3)
\]
can be realised as the natural map $\CP^\infty\lra\HP^\infty$.
In cohomology this induces a monomorphism
\[
H^*(\HP^\infty;\F_2) = \F_2[y]
                       \lra H^*(\CP^\infty;\F_2) = \F_2[x];
\quad
y\mapsto x^2.
\]
It is clear that in $H^*(-;\F_2)$, $w_2(\zeta_2)=0=w_2(\zeta_3)$
and also $w_3(\zeta_3)=0$ since $H^3(\HP^\infty)=0$.

So we can take $v=x$ and then~\eqref{eq:WSM(8.1)} becomes
\[
x^2 = y + 0x,
\]
since $t=w_2(\zeta_3)=0$. Similarly, if $p$ is an odd prime,
we have $t=0$ and the analogous relations hold in
$H^*(\CP^\infty;\F_p)$ and in $H^*(\CP^\infty;\Q)$.
\end{examp}

\section{Results on cohomology over $\F_2$}\label{sec:Coh-F2}

Now we can give some general results for the case $\k=\F_2$.
Here $H^*(-)=H^*(-;\F_2)$.

We recall Borel's theorem on the structure of Hopf algebras over
perfect fields, see~\cite{M&M}*{theorem~7.11 and proposition~7.8}.
\begin{thm}\label{thm:EMSS-u^2=0-nonilpot-2}
Suppose that the orientation $u\in H^n(M)$ satisfies $u^2=0$,
$H^*(B)$ has no nilpotents, and $\Sq^{n-1}u\neq0$. Then
$H^*(\Omega M)$ is a polynomial algebra.
\end{thm}
\begin{proof}
Let $0\neq x\in H^k(B)$ and consider $[x\.u]\in\mathrm{E}_2^{-1,k+n}$.
Then the Steenrod operation $\Sq^{n+k-1}$ satisfies
\begin{align*}
\Sq^{n+k-1}[x\.u] &= [\Sq^{n+k-1}(x\.u)] \\
                  &= [(\Sq^{k}x)\.\Sq^{n-1}u] \\
                  &= [x^2\.\Sq^{n-1}u] \neq0,
\end{align*}
since all other terms in the sum $\sum_i\Sq^ix\.\Sq^{n+k-1-i}u$
are easily seen to be trivial. It follows that the element of
$H^*(\Omega M)$ represented in the spectral sequence by $[x\.u]$
has non-trivial square since this is represented by
$\Sq^{n+k-1}[x\.u]=[x^2\.\Sq^{n-1}u]\neq0$.

More generally, using the description of the $\mathrm{E}_2$-term
in Theorem~\ref{thm:EMSS-u^2=0}, we can similarly see that an
element $[x_1\.u|\cdots|x_\ell\.u]$ with $x_i\in H^{k_i}(B)$ has
\[
\Sq^{k_1+\cdots+k_\ell+n\ell-\ell}[x_1\.u|\cdots|x_\ell\.u]
           = [x_1^2\.\Sq^{n-1}u|\cdots|x_\ell^2\.\Sq^{n-1}u]\neq0.
\]
Thus the algebra generators of $H^*(\Omega M)$ are not nilpotent,
so by Borel's theorem we see that $H^*(\Omega M)$ is a polynomial
algebra.
\end{proof}


\begin{thm}\label{thm:EMSS-u^2=0+w(n-1)=0}
Suppose that the orientation\/ $u\in H^n(M)=H^n(M;\F_2)$
satisfies $u^2=0$ and $\Sq^{n-1}u=0$. Then $H^*(\Omega M)$
is an exterior algebra.
\end{thm}
\begin{proof}
First consider an element of $w\in H^{n+k-1}(\Omega M)$ in
filtration~$1$. We can assume that this is represented in
the Eilenberg-Moore spectral sequence by $[x\.u]$ for some
$x\in H^k(B)$. Then $w^2=\Sq^{n+k-1}w$ is represented by
\[
\Sq^{n+k-1}[x\.u] = [(\Sq^k x)\.\Sq^{n-1}u] = 0,
\]
and is also in filtration $1$. Since in positive degrees,
filtration~$0$ is trivial, we have $w^2=0$.

Now we proceed by induction on the filtration~$r$. Suppose
that for every positive degree element $z\in H^*(\Omega M)$
of filtration $r\geq1$, we have $z^2=0$. Suppose that
$w\in H^*(\Omega M)$ has filtration~$r+1$. We can assume
that~$w$ is represented by $[x_1\.u|\cdots|x_{r+1}\.u]$ where
$x_j\in H^{k_j}(B)$. Applying the Steenrod operation
$\Sq^{k_1+\cdots+k_{r+1}+(r+1)n-1}$ we see that $w^2$ is also
in filtration $r+1$ and is represented by
\[
\Sq^{k_1+\cdots+k_{r+1}+(r+1)(n-1)}[x_1\.u|\cdots|x_{r+1}\.u] =
[(\Sq^{k_1}x_1)\.\Sq^{n-1}u|
              \cdots|(\Sq^{k_{r+1}}x_{r+1})\.\Sq^{n-1}u] = 0.
\]
On the other hand, the coproduct on $w$ is
\[
\psi(w) = w\otimes1 + 1\otimes w + \sum_i w_i'\otimes w''_i
\]
where the $w'_i,w''_i$ all have filtration in the range~$1$
to $r$. On squaring and using the inductive assumption we
find that
\[
\psi(w^2) = w^2\otimes1 + 1\otimes w^2,
\]
so $w^2$ is primitive and decomposable. By~\cite{M&M}*{proposition~4.21},
the kernel of the natural homomorphism
$\mathrm{P}H^*(\Omega M)\lra\mathrm{Q}H^*(\Omega M)$ consists
of squares of primitives. Since the primitives must all have
filtration~$1$, all such squares are trivial, hence $w^2=0$.
This shows that all elements of filtration $r+1$ square to
zero, giving the inductive step.

Borel's theorem now implies that $H^*(\Omega M)$ is an exterior
algebra.
\end{proof}

\section{Results on cohomology over $\F_p$ with $p$ odd}
\label{sec:Coh-Fp}

In this we give analogous results for the case $\k=\F_p$ where~$p$
is an odd prime. Here $H^*(-)=H^*(-;\F_p)$. We assume that $n$ is
odd, say $n=2m+1$, and that $M$ has an orientation class $u\in H^{2m+1}(M)$.
For degree reasons, $u^2=0$.
\begin{thm}\label{thm:EMSS-u^2=0-nonilpot-p}
Suppose that $H^*(B)$ has no nilpotents, and $\P^m u\neq0$. Then
$H^*(\Omega M)$ is a polynomial algebra.
\end{thm}
Of course $\P^m u$ defines a Wu class $W_m(\xi)$ by the formula
\[
W_m(\xi)\.u = \P^m u,
\]
and the condition $\P^m u\neq0$ amounts to its non-vanishing. The
no nilpotents condition implies that $H^*(B)$ is concentrated in
even degrees.
\begin{proof}
Let $0\neq x\in H^{2k}(B)$ and consider $[x\.u]\in\mathrm{E}_2^{-1,2k+2m+1}$.
Then the Steenrod operation $\P^{m+k}$ satisfies
\begin{align*}
\P^{m+k}[x\.u] &= [\P^{m+k}(x\.u)] \\
                  &= (\P^{k}x)\.\P^{m}u \\
                  &= x^p\.\P^{m}u \neq0,
\end{align*}
since all other terms in the sum $\sum_i\P^ix\.\P^{m+k-i}u$ are
easily seen to be trivial. It follows that the element of
$H^*(\Omega M)$ represented in the spectral sequence by $[x\.u]$
has non-trivial $p$-th power since it is represented by
\[
\P^{m+k}[x\.u]=[x^p\.\P^{m}u]\neq0.
\]
Similarly every element represented by $[x_1\.u|\cdots|x_\ell\.u]$
with $x_i\in H^{2k_i}(B)$ has non-zero $p$-th power since
\[
\P^{k_1+\cdots+k_\ell+m\ell}[x_1\.u|\cdots|x_\ell\.u] \neq 0.
\]
Thus the algebra generators of $H^*(\Omega M)$ are not nilpotent,
so by Borel's theorem we see that $H^*(\Omega M)$
is a polynomial algebra.
\end{proof}

We will say that a connective commutative graded $\F_p$-algebra
is \emph{$p$-truncated} if every positive degree element $x$
satisfies $x^p=0$. When $p=2$, being $2$-truncated is equivalent
to being exterior.
\begin{thm}\label{thm:EMSS-u^2=0+W(m)=0}
Suppose that $\P^{m}u=0$. Then $H^*(\Omega M)$ is a $p$-truncated
algebra.
\end{thm}
\begin{proof}
First consider an element of $w\in H^{2m+2k}(\Omega M)$ in filtration~$1$.
We can assume this is represented in the Eilenberg-Moore spectral
sequence by $[x\.u]\in\mathrm{E}_2^{-1,2m+2k+1}$ for some $x\in H^{2k}(B)$.
Then $w^p=\P^{m+k}w$ is represented by
\[
\P^{m+k}[x\.u] = [(\P^k x)\.\P^{m}u] = 0,
\]
and is also in filtration~$1$. Since filtration~$0$ is trivial
in positive degrees, we have $w^p=0$.

Now as in the proof of Theorem~\ref{thm:EMSS-u^2=0+w(n-1)=0},
we prove by induction on the filtration~$r$ that for every
positive degree element $z\in H^*(\Omega M)$ of filtration
$r\geq1$ has $z^p=0$. Borel's theorem now implies that every
element of $H^*(\Omega M)$ has trivial $p$-th power.
\end{proof}

\section{Rational results}\label{sec:Rational}

In this section we take $\k=\Q$. By Borel's
Theorem~\cite{M&M}*{theorem~7.11 and proposition~7.8},
we have
\begin{thm}\label{thm:loopspceCoh-poly-0}
There is an isomorphism of algebras
\[
H^*(\Omega M;\Q) \iso
\bigotimes_i \Q[x_i] \otimes \bigotimes_j \Q[y_i]/(y_j^2),
\]
where $\deg x_i$ is even and $\deg y_i$ is odd. In particular,
if\/ $H^*(M;\Q)$ is concentrated in odd degrees then
$H^*(\Omega M;\Q)$ is a polynomial algebra on even degree
generators.
\end{thm}

\section{Local to global results}\label{sec:Local-Global}

Before giving some examples, we record a variant of the local-global
result~\cite{AB&BR:qsymm}*{proposition~2.4}. We follow the convention
that a prime~$p$ can be $0$ or positive, and set $\F_0=\Q$.

Let $S\subseteq\N$ be the multiplicatively closed set generated
by a set of non-zero primes (if this set is empty then $S=\{1\}$).
Then
\[
\Z[S^{-1}] = \{a/b: a\in\Z,\; b\in S\}.
\]
In the following, whenever $p\notin S$, $\F_p=\Z[S^{-1}]/(p)$.
\begin{prop}\label{prop:PolyAlgs}
Let $H^*$ be a graded commutative connective $\Z[S^{-1}]$-algebra
which is concentrated in even degrees and with each $H^{2n}$ a
finitely generated free $\Z[S^{-1}]$-module. Suppose that for each
prime~$p\notin S$, $H(p)^*=H^*\otimes\F_p$ is a polynomial algebra,
then $H^*$ is a polynomial algebra and for every prime~$p$,
\[
\rank_{\Z[S^{-1}]}\mathrm{Q}H^{2n} = \dim_{\F_p} \mathrm{Q}H(p)^{2n}.
\]
\end{prop}
\begin{proof}
The proof of~\cite{AB&BR:qsymm}*{proposition~2.4} can be modified
by systematically replacing $\Z$ with the principal ideal domain
$\Z[S^{-1}]$ and working only with primes not contained in $S$
(including~$0$).
\end{proof}

\section{Some examples}\label{sec:Examples}

Our first example is a recasting of the main result of~\cite{AB&BR:qsymm}.
\begin{examp}\label{examp:CP}
Consider the universal line bundle $\eta\downarrow\CP^\infty$, viewed
as a real $2$-plane bundle. Then the $3$-dimensional bundle $\xi=\eta\oplus\R$
has Thom space $M\xi = \Sigma M\mathrm{U}(1)\sim\CP^\infty$. It is
straightforward to verify that the conditions of
Theorems~\ref{thm:EMSS-u^2=0-nonilpot-2} and~\ref{thm:EMSS-u^2=0-nonilpot-p}
apply. Thus $H^*(\Omega\Sigma\CP^\infty;\Z)$ is polynomial.
\end{examp}
\begin{examp}\label{examp:Spin(2/3)-cont}
Recall Example~\ref{examp:Spin(2/3)}.

Here $w_2(\zeta_3)=0=w_2(\zeta_2)$, so $H^*(\Omega M\Spin(3);\F_2)$ and
$H^*(\Omega\Sigma M\Spin(2);\F_2)$ are exterior algebras.

For an odd prime $p$, the natural map $\Sigma M\Spin(2)\lra M\Spin(3)$
induces a monomorphism in $H^*(-;\F_p)$ and in
$H^*(M\Spin(2);\F_p)=H^*(\CP^\infty;\F_p)$ we see that for the generator
$x\in H^2(\CP^\infty;\F_p)$. $\P^1x = x^p\neq0$. Therefore
$H^*(\Omega M\Spin(3);\F_p)$ and $H^*(\Omega\Sigma M\Spin(2);\F_p)$ are
polynomial algebras.

Combining these results we see that $H^*(\Omega M\Spin(3);\Z[1/2])$ and
$H^*(\Omega\Sigma M\Spin(2);\Z[1/2])$ are polynomial algebras.
\end{examp}

\section{Homology generators and a stable splitting}\label{sec:Homgens}

The map $\theta\:S_+\lra\Omega M$ introduced in Section~\ref{sec:ThomComp}
allows us to define a \emph{canonical} choice of generator $v\in H^{n-1}(S)$
in the sense of Massey's paper~\cite{Massey:CohSphBdles}, namely
\[
v = (\ev\circ\Sigma\theta)^*u.
\]
This follows from Lemma~\ref{lem:M->M}. When~$n=2m+1$ is odd,
in mod~$p$ cohomology $H^*(-)=H^*(-;\F_p)$, from~\eqref{eq:WSM(8.1)}
we obtain
\[
v^2 = s + tv,
\]
where
\[
t =
\begin{cases}
w_{2m}(\xi) & \text{if $p=2$}, \\
W_{m}(\xi) & \text{if $p$ is odd}.
\end{cases}
\]
and we define these invariants by
\begin{align*}
w_{2m}(\xi)\.u &= \Sq^{2m} u, \\
W_m(\xi)\.u    &= \P^m u.
\end{align*}
Notice that the multiplicativity given by Lemma~\ref{lem:Diag-S/M}
implies that for $x\in H^*(B)$,
\[
(\ev\circ\Sigma\theta)^*(x\.u) = xv.
\]

Now let $b_i\in H^*(B)$ form an $\F_p$-basis for $H^*(B)$,
where we suppose that $b_0=1$. Then the elements
$b_iv,b_i\in H^*(S)$ form a basis for $H^*(S)$, and the
$b_i\.u$ form a basis for $\tilde H^*(M)$. Since
\[
\delta^*(b_iv) = b_i\.u, \quad \delta^*(b_i) = 0,
\]
for the dual bases $(b_i\.v)^\circ,(b_i)^\circ$ of $H^*(S)$
and $(b_i\.u)^\circ$ of $\tilde H^*(M)$ we have
\[
\delta_*((b_i\.u)^\circ) = (b_iv)^\circ.
\]
Furthermore, $(\Sigma\theta\circ\delta)_*((b_i\.u)^\circ)$ is
dual to the class represented in the Eilenberg-Moore spectral
sequence by the primitive $[b_i\.u]$, hence the
$(\Sigma\theta\circ\delta)_*((b_i\.u)^\circ)$ form a basis
for the indecomposables $\mathrm{Q}H_*(\Omega M)$. Using
the bar resolution description of the Eilenberg-Moore
spectral sequence and the dual cobar resolution for the
homology spectral sequence
\[
\mathrm{E}^2_{*,*} = \Cotor^{H_*(M)}_{*,*}(\F_p,\F_p)
                                       \Lra H_*(\Omega M)
\]
we obtain
\begin{prop}\label{prop:H*-generators}
The homology algebra $H_*(\Omega M;\F_p)$ is the free
non-commutative algebra on the elements
$(\Sigma\theta\circ\delta)_*((b_i\.u)^\circ)$.
\end{prop}

Now we can give an analogue of the James splitting. We need
the free $S$-algebra functor $\mathbb{T}$
of~\cite{EKMM}*{section~II.4}. This is defined for an $S$-module
$X$ by
\[
\mathbb{T}X = \bigvee_{k\geq0} X^{(k)},
\]
where $(-)^{(k)}$ denotes the $k$-th smash power. The map
$\Sigma\theta\circ\delta$ gives rise to a map of spectra
\[
\Theta\:\Sigma^{-1}\Sigma^\infty M \lra \Sigma^\infty\Omega M
\]
and by the freeness property of $\mathbb{T}$, there is an
induced morphism of $S$-algebras
\[
\tilde{\Theta}\:\mathbb{T}(\Sigma^{-1}\Sigma^\infty M)
                        \lra \Sigma^\infty (\Omega M)_+,
\]
where $\Sigma^\infty(\Omega M)_+$ becomes an $S$-algebra
using the natural $A_\infty$ structure on $\Omega M$.
\begin{thm}\label{thm:StableSplitting}
Suppose that $p$ is a prime for which
\emph{Proposition~\ref{prop:H*-generators}} is true.
Then $\tilde{\Theta}$ is an $H\F_p$-equivalence of
$S$-algebras.
\end{thm}
\begin{proof}
Under the map $\tilde{\Theta}_*$, an exterior product
of classes in $H_*(\Sigma^{-k}\Sigma^\infty M^{(k)};\F_p)$
goes to their internal product in $H_*(\Omega M;\F_p)$.
Now Proposition~\ref{prop:H*-generators} shows that
$\tilde{\Theta}$ is an $\F_p$-equivalence for such
a prime~$p$.
\end{proof}
Combining our results and using an arithmetic square
argument we obtain
\begin{thm}\label{thm:LocalSplitting}
Let $S\subseteq\N$ be the multiplicatively closed set
generated by all the primes $p$ for which
\emph{Proposition~\ref{prop:H*-generators}} is false.
Then $\tilde{\Theta}$ is an $H\Z[S^{-1}]$-equivalence
of $S$-algebras. Hence there is an $H\Z[S^{-1}]$-equivalence
\[
\bigvee_{k\geq1} \Sigma^{-k}\Sigma^\infty M^{(k)}
         \lra \Sigma^\infty\Omega M.
\]
\end{thm}

Of course, this stable splitting is very different from
the James splitting for a connected based space~$X$,
\[
\Sigma \Omega\Sigma X \sim \bigvee_{k\geq1} \Sigma X^{(k)}.
\]


\begin{bibdiv}
\begin{biblist}

\bib{AB&BR:qsymm}{article}{
   author={Baker, A.},
   author={Richter, B.},
   title={Quasisymmetric functions from a topological point of view},
   journal={Math. Scand.},
   volume={103},
   date={2008},
   number={2},
   pages={208--242},
}

\bib{RB-HS:Pontryagin}{article}{
   author={Bott, R.},
   author={Samelson, H.},
   title={On the Pontryagin product in spaces of paths},
   journal={Comment. Math. Helv.},
   volume={27},
   date={1953},
   pages={320--337},
}

\bib{WGD:EMSS-Stcgce}{article}{
   author={Dwyer, W. G.},
   title={Strong convergence of the Eilenberg-Moore
   spectral sequence},
   journal={Topology},
   volume={13},
   date={1974},
   pages={255--265},
}
		
\bib{EKMM}{book}{
    author={Elmendorf, A. D.},
    author={Kriz, I.},
    author={Mandell, M. A.},
    author={May, J. P.},
     title={Rings, modules, and algebras in stable homotopy theory},
    series={Mathematical Surveys and Monographs},
    volume={47},
      note={with an appendix by M. Cole},
 publisher={American Mathematical Society},
      date={1997},
}
		
\bib{Massey:CohSphBdles}{article}{
   author={Massey, W. S.},
   title={On the cohomology ring of a sphere bundle},
   journal={Indiana Univ. Math. J. (formerly J. Math. Mech.)},
   volume={7},
   date={1958},
   pages={265--289},
}
	
\bib{M&M}{article}{
   author={Milnor, J. W.},
   author={Moore, J. C.},
   title={On the structure of Hopf algebras},
   journal={Ann. of Math. (2)},
   volume={81},
   date={1965},
   pages={211--264},
}

\bib{TP:CohLoopSpces}{article}{
   author={Petrie, T.},
   title={The cohomology of the loop spaces of Thom spaces},
   journal={Amer. J. Math.},
   volume={89},
   date={1967},
   pages={942--955},
}

\bib{NR:LoopsCones}{article}{
   author={Ray, N.},
   title={The loop group of a mapping cone},
   journal={Quart. J. Math. Oxford Ser. (2)},
   volume={24},
   date={1973},
   pages={485--498},
}

\bib{DLR:EMSS-StOps}{article}{
   author={Rector, D. L.},
   title={Steenrod operations in the Eilenberg-Moore
   spectral sequence},
   journal={Comment. Math. Helv.},
   volume={45},
   date={1970},
   pages={540--552},
}

\bib{WMS:Book}{book}{
   author={Singer, W. M.},
   title={Steenrod squares in spectral sequences},
   series={Mathematical Surveys and Monographs},
   volume={129},
   publisher={American Mathematical Society},
   date={2006},
}

\bib{LS:HomAlg&EMSS}{article}{
   author={Smith, L.},
   title={Homological algebra and the Eilenberg-Moore
   spectral sequence},
   journal={Trans. Amer. Math. Soc.},
   volume={129},
   date={1967},
   pages={58--93},
}

\bib{LS:EMSS-construction}{article}{
   author={Smith, L.},
   title={On the construction of the Eilenberg-Moore
   spectral sequence},
   journal={Bull. Amer. Math. Soc.},
   volume={75},
   date={1969},
   pages={873--878},
}

\bib{LS:LNM134}{book}{
   author={Smith, L.},
   title={Lectures on the Eilenberg-Moore spectral
   sequence},
   series={Lect. Notes in Math.},
   publisher={Springer-Verlag},
   volume={134},
   date={1970},
}

\bib{LS:KunnethThm-I}{article}{
   author={Smith, L.},
   title={On the K\"unneth theorem. I. The Eilenberg-Moore
   spectral sequence},
   journal={Math. Z.},
   volume={116},
   date={1970},
   pages={94--140},
}
		

\bib{LS:Survey}{article}{
   author={Smith, L.},
   title={On the Eilenberg-Moore spectral sequence},
   conference={
      title={Algebraic topology (Proc. Sympos. Pure Math., Vol. XXII, Univ.
      Wisconsin, Madison, Wis., 1970)},},
   book={
      publisher={Amer. Math. Soc.},
      place={Providence, R.I.},
         },
   date={1971},
   pages={231--246},
}
		
\end{biblist}
\end{bibdiv}
\end{document}